\documentclass{amsart}
\usepackage{amsfonts}
\usepackage{amsmath}
\usepackage{amsthm}
\usepackage{amssymb}
\newcommand{\mathsym}[1]{{}}

\usepackage[active]{srcltx}
\newtheorem{lemma}{Lemma}[section]
\newtheorem{theorem}[lemma]{Theorem}
\newtheorem{definition}{Definition}[section]
\newtheorem{remark}{Remark}[section]
\newtheorem{corollary}{Corollary}[section]
\numberwithin{equation}{section} 
\title[Congruence properties of induced representations]{Congruence properties of induced representations and their applications }
\author[]{Dieter Mayer}
\address{Lower Saxony Professorship\\TU Clausthal\\38678 Clausthal-Zellerfeld}
\email{dieter.mayer@tu-clausthal.de}
\author[]{Arash Momeni}
\address{Institute of Theoretical Physics\\TU Clausthal\\38678 Clausthal-Zellerfeld}
\thanks{This work was supported by the Danish National Research Foundation Centre of Excellence ''Centre for Quantum Geometry of Moduli Spaces (QGM)'' and the Volkswagenstiftung through a Lower Saxony Professorship}
\thanks{Arash Momeni would like to thank the Centre for Quantum Geometry for Moduli Spaces (QGM) for the financial support, kind hospitality and providing good working conditions.}
\email{arash.momeni@tu-clausthal.de}
\author[]{Alexei Venkov}
\address{Institute for Mathematics and Centre for Quantum Geometry (QGM)\\University of Aarhus\\ 8000 Aarhus C}
\email{venkov@imf.au.dk}
\begin{document}
\begin{abstract}
In this paper we study congruence properties of the representations $U_\alpha:=U^{PSL(2,\mathbb{Z})}_{\chi_\alpha}$ of the projective modular group ${\rm PSL}(2,\mathbb{Z})$ induced from a family $\chi_\alpha$ of characters for the Hecke congruence subgroup $\Gamma_0(4)$ basically introduced by A. Selberg. Interest in the representations $U_\alpha$ stems  from their appearance in the transfer operator approach to Selberg's zeta function for this Fuchsian group and character $\chi_\alpha$.  Hence the location of the nontrivial zeros of this function and therefore also the spectral properties of the corresponding automorphic Laplace-Beltrami operator $\Delta_{\Gamma,\chi_\alpha}$ are closely related to their congruence properties. Even if as expected these properties of $U_\alpha$ are easily shown to be equivalent to the ones well known for the characters $\chi_\alpha$, surprisingly, both the congruence and the noncongruence groups  determined by  their kernels  are quite different: those determined by $\chi_\alpha$ are character groups of type I  of the group $\Gamma_0(4)$, whereas those determined by $U_\alpha$ are such character groups of $\Gamma(4)$. Furthermore, contrary to infinitely many of the groups  $\ker \chi_\alpha $, whose noncongruence properties follow simply from  Zograf's geometric method together with Selberg's lower bound for the lowest nonvanishing eigenvalue of the automorphic Laplacian, such arguments do not apply to the groups $\ker U_\alpha$, the reason being, that they can have arbitrary genus $g\geq 0$, contrary to the groups $\ker \chi_\alpha$, which all have genus $g=0$.
\end{abstract}
\maketitle
\section{Introduction}
Whereas for congruence subgroups $\Gamma$ of the modular group ${\rm PSL}(2,\mathbb{Z})$ one expects according to the general Riemann hypothesis, that the zeros of the Selberg zeta function $Z_\Gamma (s)$ in the half plane $\Re s<\frac{1}{2}$ are located on a few lines parallel to the imaginary axis, for noncongruence 
subgroups not much is known and one even expects them to be distributed rather erratically. To understand this distribution better, A. Selberg studied in \cite{S1} the location of the poles of the Eisenstein series $E_i(z,s;\tilde{\chi}_\eta)$ for the Hecke congruence 
subgroup $\Gamma_0(4)$ and a $1$-parameter family of characters $\tilde{\chi}_\eta, -\pi/2\leq \eta\leq \pi/2$, with $\tilde{\chi}_0$ the trivial character. He got a remarkable, but often overlooked result concerning the location of the zeros of his zeta function   $Z_{\Gamma_0(4)}(s,\tilde{\chi_\eta})$,  and hence also for the location of the resonances of the automorphic Laplacian $\Delta_{\tilde{\chi}_\eta}$ in the half plane 
$\Re s <\frac{1}{2}$. When written as a dynamical zeta function the Selberg function has the form
$$Z_{\Gamma_0(4)}(s,\tilde{\chi_\eta})=\prod_{k=0}^\infty \prod_{\{\gamma\}} (1-\tilde{\chi}_\eta  (g_\gamma) e^{-(s+k)l_\gamma}),\, \Re s>1/2,$$
 with $\{\gamma\}$ the prime periodic orbits of the geodesic flow on the unit tangent bundle of the Hecke surface $\Gamma_0(4)/\mathbb{H} $ and $g_\gamma$ an hyperbolic element in $\Gamma_0(4)$ determining the periodic orbit $\gamma$. Selberg showed that there is no limitation how close to the line $\Re s =\frac{1}{2}$ for $\eta \to 0$, or how far away from it its zeros may lie for $\eta\to \pm \pi/2$. This result obviously is closely related to the fact that for each $\eta\neq 0$ the multiplicity of the continuous spectrum of the automorphic Laplacian $\Delta_{\Gamma_0(4),\tilde{\chi}_\eta }$ is equal to one and for $\eta=0$ this multiplicity is equal to three. The behaviour for $\eta\to \pm \pi/2$ on the other hand is related to the congruence properties of $\tilde{\chi}_\eta$ for these parameter values.  From 
the point 
of view of the spectral theory of automorphic functions the family of characters $\tilde{\chi}_\eta $ is singular when $\eta$ tends to $0$ .  An analogous family $\tilde{\chi}_\alpha$ of characters for the principal congruence subgroup $\Gamma(2)$, which is conjugate to the Hecke congruence subgroup $\Gamma_0(4)$, but with a slightly different normalization of the parameter $ 0\leq \alpha\leq 1$ such that $\tilde{\chi}_0=\tilde{\chi}_1\equiv 1$, was studied in their work on the existence of Maass cusp forms in non-arithmetic situations also by R. Phillips and P. Sarnak in \cite{Phillips-Sarnak}.
 Thereby they showed, obviously unaware of an older result by M. Newman \cite{Newman2}, that the character $\tilde{\chi}_\alpha$ is congruent, that means its kernel is a congruence subgroup with $\ker\tilde{\chi}_\alpha\geq \Gamma(N)$ for some $N$, iff $\alpha = j/8,\, 0\leq j\leq 8$. The explicit form of these kernels (to be precise of still another family of conjugate characters $\chi_\alpha $ of $\Gamma(2)$ with $\chi_0=\chi_1\equiv 1$) has been determined by E. Balslev and A. Venkov in \cite{Balslev-Venkov}. For irrational $\alpha$ the kernels $\ker \chi_\alpha$ are easily seen to be subgroups of infinite index in $\Gamma(2)$ and hence the characters in this case are not congruent. 
For rational $\alpha=n/d$, $d>n$ and $(n,d)=1$ on the other hand E. Balslev and A. Venkov constructed in \cite{Balslev-Venkov} a system of generators for the group $\ker\chi_\alpha$, which for $d=2, 3, \ldots$  is a cofinite, normal subgroup of $\Gamma (2)$ of index $d$ but not normal in ${\rm PSL}(2,\mathbb Z)$  \cite{Newman2}. Indeed, the group is generated by $d+2$ parabolic elements $S_1,S_2,\ldots,S_{d+2}$ with one relation
\begin{equation}
      S_1S_2\ldots S_{d+2}=Id_{2\times2}.
\end{equation}
Since for $\alpha=n/d$, $(n,d)=1$,  $\ker\chi_\alpha$ turns out to be independent of $n$ ,  it can be simply denoted by $\Gamma_d$.
Then $\Gamma_d$ has signature $(h;m_1,m_2,\ldots, m_k;p)=(0; d+2)$  \cite{Alexei}, where as usual $2h$ denotes the number of  hyperbolic, $k$  the number of elliptic and $p$  the number of  parabolic generators, respectively $m_j$  the order of the elliptic generator $e_j$.
By Gauss-Bonnet its fundamental domain $F_d$ then has hyperbolic area 
  
\begin{equation}
A_d=2\pi(2h-2+\sum_{j=1}^k (1-1/m_j)+p)=2\pi d.
\end{equation}

The groups $\Gamma_d$, as defined by Balslev and Venkov, coincide with the groups $\Gamma_{6d}$ of G. Sansone \cite{Sansone} and studied  in \cite{Newman2} by M. Newman. He solved indeed the congruence problem for these subgroups  of  $\Gamma(2)$ by  showing  $\Gamma_{6d}$ to be congruent only for $d=1,2,4,8$ with $\Gamma(2d)\leq\Gamma_{6d}$. For the character $\chi_\alpha$ of $\Gamma (2)$ this is just the above mentioned result of Phillips and Sarnak respectively Balslev and Venkov.   \newline
 In a recent paper \cite{BFM}  R. Bruggeman et al. studied in more detail the behaviour of the zeros of the Selberg zeta function $Z_{\Gamma_0(4)}(s,\chi_\alpha)$ for the group $\Gamma_0(4)$ and a family of characters  conjugate to the one of Balslev and Venkov for $\Gamma(2)$. For simplicity we denote this family again by $\chi_\alpha$. This work was initiated by numerical results in the thesis of M. Fraczek, who was able to trace the zeros of this zeta function as a function of $\alpha$ by using its 
representation in terms of the Fredholm determinant $\det(1-\mathcal{L}_{\alpha,s})$ of a family of transfer operators $\mathcal{L}_{\alpha,s}$ \cite{F, Fraczek-Mayer}. These operators are determined by the geodesic flow on the Hecke surface $\Gamma_0(4)/\mathbb{H}$ and the unitary 
representations $U_\alpha:=U^{PSL(2,\mathbb{Z})}_{\chi_\alpha}$ of the  modular group ${\rm PSL}(2,\mathbb Z)$ induced from the family $\chi_\alpha$  of characters of the Hecke congruence subgroup $\Gamma_0(4)$. One of their results is a more detailed description of Selberg's 
phenomenon of accumulation of the resonances on the critical line $\Re s=1/2$ for noncongruent values $\alpha$ of the character when approaching the trivial congruent character $\chi_0$. Also Selberg's  resonances tending on lines asymptotically 
equidistant and parallel to the real axis to $\Re s =-\infty$ for $\alpha$ approaching the congruent value $\alpha=1/4$ \cite{Balslev-Venkov1} could be confirmed numerically by Fraczek in \cite{F}.  At the same time his numerical calculations confirm  $\alpha=j/8,\, 0\leq j\leq 8$ as the only congruent values for the induced representation $U_{\chi_\alpha}$  as for the character $\chi_\alpha$. 
Indeed it is not difficult to show that, as expected, $\ker U_\alpha$ is congruent iff $\ker\chi_\alpha$ is congruent.  These two families of groups however have quite different properties: whereas the groups $\ker\chi_\alpha$, which for rational $\alpha=n/d,\, n<d$ are conjugate to the groups $\Gamma_d$ and therefore all have vanishing genus, the groups $\ker U_\alpha$ can have arbitrary large genus.  Hence, contrary to the former ones their noncongruence nature cannot be deduced simply by applying a recent geometric result of  P. Zograf \cite{Zograf3, Zograf1, Zograf2} based on previous results of Yang and Yau \cite{Y.-Y.} respectively Hersh \cite{Hersch}, together with A. Selberg's famous lower bound for the eigenvalues of congruence subgroups \cite{Selberg65}. Instead one has to use some different algebraic approach based on Wohlfahrt's notion of level for general subgroups of the modular group. That there are infinitely many  noncongruence groups even in any fixed genus $g\geq0$ has been shown by G. Jones \cite{Jones}.\newline
Many of the known examples of noncongruence subgroups are so called character groups, that means, appear as the kernel of some group 
epimorphism of a  congruence subgroup onto some finite abelian group or equivalently, as a normal finite index subgroup of a 
congruence group with abelian quotient. Most such examples are kernels of unitary characters like the lattice groups of Rankin \cite{Ran} or 
the noncongruence groups of Klein and Fricke \cite{Fri}. Such character groups play an important role in the theory of modular forms for 
noncongruence subgroups and there especially for the Atkin-Swinnerton-Dyer congruence relations (see for instance the series of papers 
\cite{LLY}, \cite{ALL}, \cite{Kurth-Long}, and \cite{Long}). Obviously the groups $\Gamma_d$ for noncongruence values of $d$ belong to this class of 
noncongruence groups. On the other hand  the group $\ker U_\alpha$ for a noncongruence rational value of $\alpha$ 
will be shown to be the kernel of an epimorphism $\phi_\alpha: \Gamma(4)\to A_\alpha$, where $A_\alpha$ is a finite abelian group freely 
generated by three six-dimensional matrices each of order $N(\alpha)$, which for $\alpha=p/q$ and $(p,q)=1$ is given by $N(\alpha)= \
min 
\{n: q | 4 n\}$. To our knowledge not many examples of such noncongruence character groups arising as kernels of 
epimorphisms onto non-cyclic finite abelian groups have been discussed in the literature.\newline
The paper is organized as follows. In chapter \ref{ch:ind} we recall Selberg's character $\chi_\alpha$ for the group $\Gamma_0(4)$ and the $6$-dim. monomial representation $U_\alpha$ of the modular group ${\rm PSL}(2,\mathbb Z)$
induced from $\chi_\alpha$.  We  briefly recall Selberg's upper bound for the  smallest eigenvalue $\lambda_1$ of the automorphic Laplace-Beltrami operator $\Delta_\Gamma$ for congruence subgroups $\Gamma$ and Zografs lower bound for this eigenvalue for general finite index subgroups of the modular group.  We show how 
these results imply the noncongrunce property of infintely many of the groups $\Gamma_d$ defined by the kernel $\ker \chi_\alpha$ for $\alpha=n/d$. In chapter \ref{ch:image} we introduce the groups $G_\alpha=U_\alpha({\rm PSL}(2,\mathbb Z))$, $0\leq\alpha\leq 1$  
and  relate them to the group $G_0$.  We determine the kernel $\ker U_0$ of the representation $U_0$ as the principal congruence subgroup $\Gamma(4)$ and show that $\ker U_\alpha \leq \Gamma(4)$. We introduce the groups 
$A_\alpha=U_\alpha(\ker U_0)$ which are finitely generated abelian normal subgroups of $G_\alpha$ and of finite 
order iff $\alpha$ is rational. The factor group $G_\alpha/A_\alpha$ on the other hand is isomorphic to the group $G_0$,  which itself by a Theorem of M. Millington \cite{Mi} is isomorpic to the modulary group $G(4)={\rm PSL}(2,\mathbb{Z})/\Gamma(4)$. Chapter \ref{ch:congruence} 
contains the main results of this paper: namely a complete characterisation for rational $\alpha$ of the groups $\ker U_\alpha$ by determining their index in $\Gamma(4)$, their genus, the number of their cusps, their Wohlfahrt level and the number of their free generators. We show that $\ker U_\alpha$ is congruent iff the character $\chi_\alpha$ is congruent. Independently from this result, using the above group datas  leads to another determination of those $\alpha$-values for which $\ker U_\alpha$
is congruent. Indeed for these values the congruence group $\ker U_\alpha$ either coincides with the  principle subgroup $\Gamma(4)$ or with $\Gamma(8)$ depending just on the order of the abelian group $\Gamma(4)/\ker U_\alpha$.

\section{Selberg's character $\chi_\alpha$ for $\Gamma_0(4)$ and the  induced representation $U_\alpha$ of ${\rm PSL}(2,\mathbb Z)$}
The projective modular group ${\rm PSL}(2,\mathbb Z)$ is defined as   
\begin{equation}
{\rm PSL}(2,\mathbb Z)=\left\lbrace \left( \begin{array}{cc}
a&b\\
c&d\\
\end{array}
\right)\mid a,b,c,d\in \mathbb Z,ad-bc=1\right\rbrace/\left\lbrace  \pm Id\right\rbrace.
\end{equation}
This group is generated by the elements 
\begin{equation}
T=\pm\left( \begin{array}{cc}
1&1\\
0&1\\
\end{array}
\right)\qquad \rm{and}\qquad\quad S=\pm\left( \begin{array}{cc}
0&-1\\
1&0\\
\end{array}
\right),
\end{equation}
with the relations $S^2=(ST)^3=\pm Id$.
Many of the noncongruence subgroups of the modular group ${\rm PSL}(2,\mathbb Z)$ discussed in the literature are so called  character groups of congruence subgroups $G$ of the modular group . Quite generally one calls a subgroup $\Gamma$ of an arithmetic subgroup $G$, that is $G\leq{\rm PSL}(2,\mathbb Z)$ of finite index, a character group of $G$ if $\Gamma=\ker\phi$ for some group epimorphism $\phi: G\rightarrow A$ onto a finite abelian group $A$. 
The character group $\Gamma$ is of type $\rm I$ if $\phi(g)\neq id$ for some parabolic element $g\in G$, otherwise it is of type $\rm I\rm I$. The groups $\Gamma_d$ are obviously character groups of the congruence subgroup $\Gamma_0(4)\leq {\rm PSL}(2,\mathbb Z)$ of type $\rm I$, since $\Gamma_d=\ker \chi_\alpha$, $\alpha=\frac{n}{d}$, $n<d$, $(n,d)=1$.

The principal congruence subgroup $\Gamma(n)$ of level $n$ is defined as
\begin{equation}\label{pcs}
\Gamma(n):=\left\lbrace \gamma\in {\rm PSL}(2,\mathbb Z)\vert \gamma\equiv\pm\left( \begin{array}{cc}
1&0\\
0&1\\
\end{array}
\right)\mod n \right\rbrace.
\end{equation}
The index $\mu_{\Gamma(n)}$ of $\Gamma(n)$ in ${\rm PSL}(2,\mathbb Z)$ is given by (see for example \cite{Rankin})
\begin{equation}
\mu_{\Gamma(n)}=\left[ {\rm PSL}(2,\mathbb Z):\Gamma(n)\right]=\begin{cases} \dfrac{1}{2}n^3\prod_{p\vert n}(1-\dfrac{1}{p^2}) \quad \text{for}\; n>2\\
6 \quad \quad \quad \quad \quad \quad \quad \quad \;\text{for}\; n=2 \end{cases}.
\end{equation}
where $p$ runs over all primes dividing $n$.
The Hecke congruence subgroup $\Gamma_0(n)$ of ${\rm PSL}(2,\mathbb Z)$ on the other hand is defined by
\begin{equation}
\Gamma_0(n)=\left\lbrace \gamma\in {\rm PSL}(2,\mathbb Z)\,\vert\,c=0 \mod n\right\rbrace.
\end{equation}
Its index $\mu_{\Gamma_0(n)}$ in ${\rm PSL}(2,\mathbb Z)$ is given by 
\begin{equation}
\mu_{\Gamma_0(n)}=n \prod_{p\vert n} (1+\dfrac{1}{p})
\end{equation}
In the following we are mostly interested in the case $n=4$. In this case $\Gamma_0(4)$ is freely generated by
\begin{equation}
T=\pm\left( \begin{array}{cc}
1&1\\
0&1\\
\end{array}
\right)\qquad {\rm and}\qquad\quad ST^4S=\pm\left( \begin{array}{cc}
-1&0\\
4&-1\\
\end{array}
\right).
\end{equation}
and $\mu_{\Gamma_0(4)}= 6$.
As representatives of the right cosets $\Gamma_0(4)\backslash{\rm PSL}(2,\mathbb Z)$ of $\Gamma_0(4)$ in ${\rm PSL}(2,\mathbb Z)$ we choose the following set $R$ of elements of ${\rm PSL}(2,\mathbb Z)$
\begin{equation}\label{Rep.4}
 R=\left\lbrace  Id, S, ST, ST^2, ST^3, ST^2S\right\rbrace.
\end{equation}
Selberg's character 
\begin{equation}
\chi_\alpha:\Gamma_0(4)\rightarrow {\rm Aut} \,\mathbb C, \quad 0\leq\alpha\leq1,
\end{equation}
for the group $\Gamma_0(4)$ is defined by the following assignments to the above generators:
\begin{equation}
 \chi_\alpha(T)=\exp(2\pi i\alpha),\quad \chi_\alpha(ST^4S)=1.
\end{equation}
This corresponds to the family of conjugate characters  considered by Balslev and Venkov for the conjugate group $\Gamma(2)$ in \cite{Balslev-Venkov}. Since $\chi_0=\chi_1$ and $\chi_{1-\alpha}=\chi_\alpha^*$ we can restrict ourselves in the following to the parameter range $0\leq \alpha\leq 1/2$.
 The representation $U_\alpha:= U^{\rm PSL(2,\mathbb Z)}_{\chi_\alpha}$ of $\rm PSL(2,\mathbb Z)$ induced from Selberg's character $\chi_\alpha$ for $\Gamma_0(4)$ for our choice $R$ of representatives  of $\Gamma_0(4)\backslash{\rm PSL}(2,\mathbb Z)$ is then given by
\begin{equation}\label{rep.-U}
\left[ U_\alpha(g)\right]_{i,j}=\delta_{\Gamma_0(4)}(r_i g r_j^{-1}) \chi_\alpha (r_i g r_j^{-1}),\quad r_i\in R,\quad 1\leq i,j \leq6
\end{equation}
where
\begin{equation}\label{delta}
\delta_{\Gamma_0(4)}(\gamma)=\begin{cases}
                                          1&\gamma\in\Gamma_0(4),\\
					  0&\gamma\not\in\Gamma_0(4).
                                         \end{cases}
\end{equation}
Obviously $U_\alpha(g)$ is a $6$-dimensional monomial matrix, that means, has only one nonvanishing entry in every row and column. For $\alpha=0$ the matrix $U_0(g)$ becomes  a permutation matrix with all nonvanishing entries identical one, that is \begin{equation}\label{indtriv}
 \left[U_0(g)\right]_{ij}=\delta_{\Gamma_0(4)}(r_igr_j^{-1}),\quad r_i\in R,\quad 1\leq i,j\leq6.
\end{equation}

For the generators $S$ and $T$ of $\rm PSL(2,\mathbb Z)$ one then finds : 
\begin{equation}
U_\alpha(S)=\left(
\begin{array}{cccccc}
 0 & 1 & 0 & 0 & 0 & 0 \\
 1 & 0 & 0 & 0 & 0 & 0 \\
 0 & 0 & 0 & 0 & \exp(-2\pi i\alpha) & 0 \\
 0 & 0 & 0 & 0 & 0 & 1 \\
 0 & 0 & \exp(2\pi i\alpha) & 0 & 0 & 0 \\
 0 & 0 & 0 & 1 & 0 & 0
\end{array}
\right)\end{equation}
and
\begin{equation}U_\alpha(T)=\left(
\begin{array}{cccccc}
 \exp(2\pi i\alpha)& 0 & 0 & 0 & 0 & 0 \\
 0 & 0 & 1 & 0 & 0 & 0 \\
 0 & 0 & 0 & 1 & 0 & 0 \\
 0 & 0 & 0 & 0 & 1 & 0 \\
 0 & 1 & 0 & 0 & 0 & 0 \\
 0 & 0 & 0 & 0 & 0 & \exp(-2\pi i\alpha)
\end{array}
\right).
\end{equation}

\begin{definition}
The character $\chi_\alpha$ respectively the representation $U_\alpha$ is a congruence character respectively a congruence representation or shortly congruent iff its kernel $\ker \chi_\alpha$ respectively $\ker U_\alpha$ is a congruence subgroup, that means $\Gamma(N)\leq \ker \chi_\alpha$ respectively $\Gamma(N)\leq \ker U_\alpha$ for some $N$.
\end{definition}
It is well known from the work of M. Newman \cite{Newman2} and later reproved by several other authors, that only for $\alpha = j/8,\, 0\leq j \leq 4$ the character $\chi_\alpha$ is congruent. Obviously, for $\alpha$ irrational the group $\ker \chi_\alpha$ has infinite index in the modular group and hence cannot be congruent. On the other hand  the group $\ker \chi_\alpha$ for rational $\alpha=n/d, \, 1\leq n\leq d-1$, does not depend on $n$ and can hence can be denoted by $\Gamma_d$. Balslev and Venkov showed in \cite{Balslev-Venkov}, that $\Gamma_d$ has vanishing genus and the area $A_d$ of its its fundamental domain $F_d $ is given by $A_d= 2\pi d$.  That  $\Gamma_d$ can be congruent only for finitely many values of $d$ follows indeed already from a remarkable geometric result of P. Zograf  \cite{Zograf1}, \cite{Zograf2} based on previous work of Yang and Yau \cite{Y.-Y.} respectively Hersh \cite{Hersch}, together with A. Selberg's famous lower bound for the eigenvalues of the automorphic Laplacian $\Delta_\Gamma$ for  congruence subgroups $\Gamma$ \cite{Selberg65} (for more recent lower bounds see \cite{Sar}). Let us briefly recall these results:
\label{ch:ind}
\begin{theorem}[Zograf]\label{tttt234}
Let $\Gamma$ be a discrete cofinite subgroup of ${\rm PSL}(2,\mathbb R)$ of signature $(h;m_1,m_2,\ldots, m_k;p)$ and genus $g$.  Let $A(F)$ be the hyperbolic area of its fundamental domain $F$. Assume $A(F)\geq 32 \pi(g+1)$. Then the set of eigenvalues of the automorphic Laplacian $\Delta_\Gamma$ in $(0,1/4)$ is not empty and
\begin{equation}
\lambda_1 <\dfrac{8\pi(g + 1)}{A(F)}
\end{equation}
where $\lambda_1>0$ is the smallest non zero eigenvalue of $\Delta_\Gamma$.
\end{theorem}
On the other hand Selberg proved the following lower bound for the smallest non vanishing eigenvalue $\lambda_1$ for any congruence subgroup 
\begin{theorem}[Selberg]
Let $\Gamma$ be a congruence subgroup of ${\rm PSL}(2,\mathbb Z)$. Then
\begin{equation}
      (3/16)\leq\lambda_1.
\end{equation}

\end{theorem}
Selberg's sharper eigenvalue conjecture for congruence subgroups is indeed $\lambda_1\geq 1/4$ (see \cite{Selberg65}). Notice that the interval $[0,1/4)$ is free from the continuous spectrum of the automorphic Laplacian $\Delta_\Gamma$, which is real and given by $[1/4,\infty)$.
If we now combine these two theorems, we get for congruence subgroups
\begin{equation}\label{zzz}
3/16 < \dfrac{8\pi( g + 1)}{A(F)}
\end{equation}
If we assume, that for a given $d$ the group $\Gamma_d$, which has vanishing genus $g$, is a congruence subgroup, we get from \eqref{zzz} that $3/16 <8\pi/2\pi d$ or $d < 64/3$ and hence there are only finitely many $d$ with $\Gamma_d$ a congruence subgroup.

\section{The groups $U_\alpha({\rm PSL}(2,\mathbb Z)$ and $U_\alpha(\Gamma(4))$} \label{ch:image}
Denote by $G_\alpha$ the group $G_\alpha=U_\alpha({\rm PSL}(2,\mathbb Z))$ determined by the induced representation $U_\alpha$. Since ${\rm PSL}(2,\mathbb Z)$ is generated by $S$ and $T$, the group $G_\alpha$ is generated by $U_\alpha(S)$ and $U_\alpha(T)$. Furthermore we denote by $M(6,\mathbb C)$ the group of all monomial matrices and by $\Delta(6,\mathbb C)$ the group of all diagonal matrices in $\rm GL(6,\mathbb C)$.
It is well known that $M(6,\mathbb C)$ is the normalizer of $\Delta(6,\mathbb C)$ in $\rm GL(6,\mathbb C)$ (see for instance \cite{Alperin}, page 48, Exercise 7). Hence, $\Delta(6,\mathbb C)$ is obviously normal in $M(6,\mathbb C)$ . Denote furthermore by $W$ the set of all $6$- dimensional permutation matrices in $\rm GL(6,\mathbb C)$. It is a subgroup of $\rm GL(6,\mathbb C)$ which is also called the Weyl group. The group $W$ is obviously isomorphic to $S_6$, the symmetric group of degree $6$. Then the group $M(6,\mathbb C)$ of monomial matrices in dimension $6$ has the following semidirect product structure (see \cite{Alperin}, page 48, Exercise 7)
\begin{equation}\label{semi}
M(6,\mathbb C)=\Delta(6,\mathbb C)\rtimes W,
\end{equation}
and therefore each element $m\in M(6,\mathbb C)$ can be uniquely expressed as $m=\delta w$ where $\delta\in\Delta(6,\mathbb C)$ and $w\in W$ .

Since the generators $U_\alpha(S)$ and $U_\alpha(T)$ of $G_\alpha$ obviously belong to $M(6,\mathbb C)$ the group $G_\alpha$ is a subgroup of $M(6,\mathbb C)$
\begin{equation}
 G_\alpha\leq M(6,\mathbb C).
\end{equation}
 Then we can show
\begin{lemma}\label{lem:decom}
Let $U_0$ be the representation of $\rm PSL(2,\mathbb Z)$ induced from the trivial character $\chi_0$ of $\Gamma_0(4)$. Then each element $U_\alpha(g)\in G_\alpha$ has a unique representation as
\begin{equation}
 U_\alpha(g)=D_\alpha(g)U_0(g),
\end{equation}
where $D_\alpha(g)\in \Delta(6,\mathbb C)$. 
\end{lemma}
\begin{proof}
For $g\in{\rm PSL}(2,\mathbb Z)$ denote by $D_\alpha(g)\in\Delta(6,\mathbb C)$ the diagonal matrix  
\begin{equation}\label{diag1}
\left[ D_\alpha(g)\right] _{ik}=\delta_{ik}\chi_\alpha(r_ig r(i)^{-1}),\quad1\leq i,k\leq6.
\end{equation}
Here, $r_i$ and $r(i)$ are elements of the set $R$ of representatives of $\Gamma_0(4)\backslash{\rm PSL}(2,\mathbb Z)$ with $r(i)$ uniquely determined by the condition $r_ig r(i)^{-1}\in\Gamma_0(4)$.
Then we have
\begin{equation}
 \left[ D_\alpha(g)U_0(g)\right] _{ij}=\sum_{k=1}^6 \left[ D_\alpha(g)\right] _{ik} \left[U_0(g)\right] _{kj}.
\end{equation}
Inserting \eqref{diag1} this reads
\begin{equation}
 \left[ D_\alpha(g)U_0(g)\right] _{ij}=\chi_\alpha(r_ig r(i)^{-1}) \left[U_0(g)\right] _{ij}.
\end{equation}
But according to the definition of $U_0$ in \eqref{indtriv}, we get 
\begin{equation}
 \left[ D_\alpha(g)U_0(g)\right] _{ij}=\chi_\alpha(r_ig r(i)^{-1}) \delta_{\Gamma_0(4)}(r_ig r_j^{-1}).
\end{equation}
Hence
\begin{equation}
  \left[ D_\alpha(g)U_0(g)\right] _{ij}=\begin{cases}
                 \chi_\alpha(r_ig r_j^{-1})&{\rm if}\,\,r_ig r_j^{-1}\in\Gamma_0(4),\\
		 0&{\rm if}\,\,r_ig r_j^{-1}\not\in\Gamma_0(4)
                \end{cases}
\end{equation}
or
\begin{equation}
\left[ D_\alpha(g)U_0(g)\right] _{ij}=\delta_{\Gamma_0(4)}(r_ig r_j^{-1}) \chi_\alpha(r_ig r_j^{-1}) = U_\alpha(g)_{i,j}.
\end{equation}
Since $U_0(g)$ is a permutation matrix in $W$ and $G_\alpha$ is a subgroup of $M(6,\mathbb C)$, this decomposition according to \eqref{semi} is unique.
\end{proof}

Since $\Delta(6,\mathbb C)$ is normal in $M(6,\mathbb C)$ and $G_\alpha\leq M(6,\mathbb C)$, the group 
 $A_\alpha:=G_\alpha\cap\Delta(6,\mathbb C)$ is normal in $G_\alpha$. 
By definition, $A_\alpha$ is the group of all diagonal matrices in $G_\alpha$. Hence, according to lemma \ref{lem:decom}, $A_\alpha$ is the image of the kernel $\ker U_0$  of the representation $U_0$ under the map $U_\alpha$, that is
\begin{equation}\label{norsub}
A_\alpha=\left\lbrace  U_\alpha(\gamma)\vert \gamma\in\ker U_0\right\rbrace.
\end{equation}

\begin{lemma}\label{ker3}
Let $U_0$ be the representation of $\rm PSL(2,\mathbb Z)$ induced from the trivial  character $\chi_0$ of $\Gamma_0(4)$. Then 
\begin{equation}
\ker U_0=\left\lbrace g\in {\rm PSL}(2,\mathbb Z)\vert U_0(g)=Id_{6\times6}\right\rbrace = \Gamma(4).
\end{equation}

\end{lemma}
\begin{proof}
Since $\Gamma(4)\lhd{\rm PSL(2,\mathbb Z)}$, for each $\gamma\in \Gamma(4)$ and $r\in R$ there exists a $\gamma'\in \Gamma(4)$ such that $r\gamma r^{-1}=\gamma'$. Thus, according to \eqref{indtriv}  $\Gamma(4)\leq\ker U_0$. To show $\ker U_0\leq\Gamma(4)$ take $\gamma\in \ker U_0$. Then, according to \eqref{indtriv}  for each $r\in R$ one has $r\gamma r^{-1}\in\Gamma_0(4)$. Since $Id\in R$, necessarily $\gamma\in \Gamma_0(4)$. But for $\gamma\in \Gamma_0(4)$ one has $S\gamma S^{-1}\in \Gamma^0(4)$. On the other hand  $S\in R$ and therefore $S\gamma S^{-1}\in \Gamma_0(4)$. Hence $\gamma$ itself must belong to $ \Gamma_0(4)\cap\Gamma^0(4)$.  Conjugating $\gamma=\left(
\begin{array}{cc}
 a & 4b\\
 4c & d\\
\end{array}
\right)\in  \Gamma_0(4)\cap\Gamma^0(4)$ by $ST\in R$ shows that  $ST\gamma(ST)^{-1}\in\Gamma_0(4)$ iff $a\equiv d\mod4$. This and the fact that $\det \gamma=1$ yields $\gamma\in \Gamma(4)$. Hence $\ker U_0\leq\Gamma(4)$. This completes the proof.
\end{proof}
 \begin{remark}
 Obviously $\ker U_0$ is given by the maximal normal subgroup of the modular group which is contained in $\Gamma_0(4)$. Since one shows quite generally that the maximal normal subgroup $H(n)\lhd{\rm PSL(2,\mathbb Z)}$ with $H(n)\leq \Gamma_0(n)$ is given by $$H(n)=\{g\in {\rm PSL}(2,\mathbb{Z}): g = \pm \begin{pmatrix}\alpha&0\\ 0&\alpha\end{pmatrix} \mod n, \quad\alpha^2=1\mod n\}$$
and since in the case $n=4$ there is only the solution $\alpha=\pm 1$ one finds  $H(4)=\Gamma(4)$. 
 \end{remark}\begin{corollary}\label{iso2}
The normal subgroup $A_\alpha$ of $G_\alpha$ in \eqref{norsub} is given by
\begin{equation}
A_\alpha=\left\lbrace  U_\alpha(\gamma)\vert \gamma\in\Gamma(4)\right\rbrace.
\end{equation}
\end{corollary}

According to this corollary the generators of $A_\alpha$ can be calculated explicitly from generators of $\Gamma(4)$. A set of generators of $\Gamma(4)$ is given for instance by \cite{Kiming-et-al}
\begin{eqnarray}\label{sdf2}
g_1=T^4=\left(
\begin{array}{cc}
 1 & 4\\
 0 & 1\\
\end{array}
\right),\nonumber\\g_2=ST^{-4}S=\left(
\begin{array}{cc}
 1 & 0\\
 4 & 1\\
\end{array}
\right),\nonumber\\g_3=T^{-1}ST^4ST=\left(
\begin{array}{cc}
 -5 & -4\\
 4 & 3\\
\end{array}
\right),\nonumber\\g_4=T^{-2}ST^{-4}ST^{-2}=\left(
\begin{array}{cc}
 7 & -12\\
 -4 & 7\\
\end{array}
\right),\nonumber\\g_5=TST^{-4}ST^{-1}=\left(
\begin{array}{cc}
 -5 & 4\\
 -4 & 3\\
\end{array}
\right).
\end{eqnarray}
The corresponding generators of $A_\alpha$ are obtained by calculating their induced representations $U_\alpha (g_i)$ :
\begin{equation}\label{genA1}
A_1(\alpha):=U_\alpha(g_1)={\rm diag}( \exp(8\pi i \alpha),1,1,1,1, \exp(-8\pi i \alpha)),
\end{equation}
\begin{equation}\label{genA2}
A_2(\alpha):=U_\alpha(g_2)={\rm diag}(1,\exp(-8\pi i \alpha),1,\exp(8\pi i \alpha),1,1),
\end{equation}
\begin{equation}\label{genA3}
A_3(\alpha):=U_\alpha(g_3)={\rm diag}(1,1,\exp(8\pi i \alpha),1,\exp(-8\pi i \alpha),1)
\end{equation}
where ${\rm diag}(a_1,\ldots,a_6)$ denotes the $6$-dimensional diagonal matrix with entries $\{a_i\}$.
It turns out, that
\begin{equation}\label{gen}
U_\alpha(g_5)=U_\alpha(g_3),\quad U_\alpha(g_4)=\left[ U_\alpha(g_1)U_\alpha(g_2)\right]^{-1},
\end{equation}
and hence the group $A_\alpha$ is generated by the three elements $A_k(\alpha), \, 1\leq k \leq 3$
\begin{equation}\label{abelA}
A_\alpha=\left\langle A_1(\alpha),A_2(\alpha),A_3(\alpha)\right\rangle. 
\end{equation}

Next we consider the factor group $G_\alpha/A_\alpha$. 
\begin{lemma}\label{iso}
The factor group $G_\alpha/A_\alpha$ is isomorphic to the modulary group $G(4)={\rm PSL}(2,\mathbb Z)/\Gamma(4)$.
\end{lemma}
\begin{proof}
Since $G_\alpha = U_\alpha (PSL(2,\mathbb{Z}))$ and $A_\alpha= U_\alpha(\Gamma(4))$ it follows that $G_\alpha\cong {\rm PSL}(2,\mathbb Z)/ \ker U_\alpha$ respectively $A_\alpha\cong \Gamma(4)/ \ker U_\alpha$. Hence 
\begin{equation}
G_\alpha/A_\alpha\cong\rm{PSL}(2,\mathbb Z)/\Gamma(4)=G(4).
\end{equation}
\end{proof}

\section{Noncongruence character groups and the induced representation $U_\alpha$}\label{ch:congruence}
According to the definition of $U_\alpha$ in \eqref{rep.-U}  an element $g\in \rm PSL(2,\mathbb Z)$ belongs to $\ker U_\alpha$ iff for all representatives $r\in R$  in \eqref{Rep.4} one has $\delta_{\Gamma_0(4)} (rgr^{-1})\chi_\alpha(rgr^{-1})=1$ and hence iff  $rgr^{-1}\in\ker\chi_\alpha$ for all $r\in R$, that is 
\begin{equation}\label{jk09}
\ker U_\alpha=\left\lbrace g\in {\rm PSL} (2,\mathbb Z) \,\vert\, rgr^{-1}\in\ker\chi_\alpha,\forall r\in R\right\rbrace.
\end{equation}
From this one concludes
\begin{lemma}\label{theorem:cong}
$\ker U_\alpha$ is congruence if and only if $\ker \chi_\alpha$ is congruence. 
\end{lemma}
\begin{proof}
Since $Id\in R$, according to \eqref{jk09} $\ker U_\alpha$ is a subgroup of $\ker \chi_\alpha$
\begin{equation}\label{nm67}
 \ker U_\alpha\leq\ker \chi_\alpha.
\end{equation}
Thus, if $\ker U_\alpha$ is a congruence subgroup then also $\ker\chi_\alpha$ is a congruence subgroup.
To prove the converse, consider the kernel $\ker U_\alpha$ in \eqref{jk09}, which is given by the following intersection of sets
\begin{equation}\label{ker2}
\ker U_\alpha= r_1^{-1} \ker \chi_\alpha r_1\cap r_2^{-1} \ker \chi_\alpha r_2\cap\ldots \cap r_6^{-1} \ker \chi_\alpha r_6,\quad r_i\in R.
\end{equation}
If $\ker \chi_\alpha$ is congruence, then $\Gamma(n)\leq\ker\chi_\alpha$ for some $n\in \mathbb N$. But $\Gamma(n)$ is normal in ${\rm PSL
}(2,\mathbb Z)$ and therefore $\Gamma(n)\leq r^{-1}\ker\chi_\alpha \, r$ for all $r\in R$. Therefore, according to \eqref{ker2}, $\Gamma(n)\leq\ker U_\alpha$. Hence $\ker U_\alpha$ is also a congruence subgroup.
\end{proof}

As a corollary of theorem \ref{theorem:cong} we then have
\begin{corollary}
The representation $U_\alpha$ is congruence iff Selberg's character $\chi_\alpha$ is congruence. 
\end{corollary}

Next we are going to determine several properties of $\ker U_\alpha$ which will lead us for rational non-congruence values of $\alpha$ to an  infinite family of noncongruence character groups  with arbitrary large genus, rather different from the ones determined by the character $\chi_\alpha$. At the same time this provides us with an independent determination of the congruence values $\alpha$ of Newman et al. for the character $\chi_\alpha$ respectively the representation $U_\alpha$ and the corresponding congruence groups.\newline
For this denote by $N=N(\alpha)$ the order of the generators of the group $A_\alpha$ defined in \eqref{abelA}.
According to  Corollary \ref{iso2}
\begin{equation}\label{rty}
\Gamma(4)/\ker U_\alpha\cong A_\alpha,
\end{equation}
and hence the index $\mu(\alpha)=\left[{\rm PSL(2,\mathbb Z)}:\ker U_\alpha\right]$ of $\ker U_\alpha$ in ${\rm PSL(2,\mathbb Z)}$ is equal to the number of elements of $A_\alpha$ times the index of $\Gamma(4)$ in ${\rm PSL}(2,\mathbb Z)$ .  Thus we have
\begin{equation}\label{indexalpha}
 \mu(\alpha)=24 N^3=24N(\alpha)^3.
\end{equation}
For irrational $\alpha$ the subgroup $\ker U_\alpha$ is therefore of infinite index in ${\rm PSL}(2,\mathbb Z)$ and cannot be a congruence group. 
In the following let $\alpha$ be rational with $N(\alpha)=N$ for some $N\in \mathbb N$.\newline
Using the Gauss-Bonnet formula we can determine the number of generators of $\ker U_\alpha$.
Since $\ker U_\alpha\leq \Gamma(4)$ it has no elliptic elements. The Gauss-Bonnet formula for a group $\Gamma$ without elliptic elements then reads  \cite{Alexei} (page 15),
\begin{equation}
\vert F\vert=2\pi(2g-2+p), 
\end{equation}
where $\vert F\vert$ is the area of the fundamental domain of $\Gamma$, $g$ is its genus and $p$ the number of its cusps.
It is also known that the number of generators of $\Gamma$ is given by $2g+p$ \cite{Alexei} (page 14). But for the group $\ker U_\alpha$ one has also
\begin{equation}
\vert F\vert=\mu(\alpha)\dfrac{\pi}{3},
\end{equation}
where $\pi/3$ is the area of the fundamental domain of ${\rm PSL}(2,\mathbb Z)$ and $\mu(\alpha)$ is the index of $\ker U_\alpha$ in ${\rm PSL(2,\mathbb Z)}$ determined in \eqref{indexalpha}. Hence the number of generators of $\ker U_\alpha$ is given by 
\begin{equation}\label{ggg}
2g+p=4N^3+2.
\end{equation}
The number $\mathcal{N}(\alpha)$ of free generators on the other hand is given by \cite{Alexei} (page 14)
\begin{equation}
\mathcal N(\alpha)=2g+p-1=4N^3+1.
\end{equation}

Let us next recall the concept of the width of a cusp  respectively Wohlfahrt's generalized notion of the level of any subgroup $\Gamma$ of the modular group \cite{Wohlfahrt}: 
\begin{definition}
For $x\in\mathbb Q\cup\left\lbrace \infty\right\rbrace$ a cusp of the group  $\Gamma\leq{\rm PSL}(2,\mathbb Z)$ and $\sigma\in {\rm PSL}(2,\mathbb Z)$ with $\sigma \infty=x$, let $P\in\Gamma$ be a primitive parabolic element with $Px=x$. 
If \begin{equation}
 \sigma P\sigma^{-1}=\left(
\begin{array}{cc}
 1 & m\\
 0 & 1\\
\end{array}
\right)\in  {\rm PSL}(2,\mathbb Z).
\end{equation}
then $\vert m\vert$ is called the width of the cusp $x$ of $\Gamma$.
\end{definition}

\begin{definition}
Let $\Gamma\leq{\rm PSL}(2,\mathbb Z)$ and $W(\Gamma)\leq\mathbb N$ be the set of widths of the cusps of $\Gamma$. If $W(\Gamma)$ is nonempty and bounded in $\mathbb N$, the Wohlfahrt level $n(\Gamma)$ of $\Gamma$ is defined to be the least common multiple of the elements of $W(\Gamma)$. Otherwise the level is defined to be zero.
\end{definition}
For congruence subgroups $\Gamma$ F. Klein on the other hand defined the level as follows \cite{Wohlfahrt} 
\begin{definition}
The level of a congruence subgroup is defined to be the least integer $n$ such that $\Gamma(n)\subset\Gamma$.
\end{definition}
It is known that for congruence subgroups Wohlfahrt's and F. Klein's definitions of the level coincide \cite{Wohlfahrt}, that is
If $\Gamma$ is a congruence subgroup of Wohlfahrt level $n$ then $\Gamma(n)\leq\Gamma$. 
Next we determine the Wohlfahrt level of the group $\ker U_\alpha$. Since $\ker U_\alpha$ is normal in $ {\rm PSL}(2,\mathbb Z)$ all cusps of $\ker U_\alpha$ have the same width \cite{Alexei} (page 160). Thus it is enough to find the width just for one cusp.  
According to \eqref{genA1}, for $\alpha$ with $N(\alpha)=N$, we have $U_\alpha(g_1)^N=Id_{6\times6}$ for the generator
\begin{equation}
g_1=\left(
\begin{array}{cc}
 1 & 4\\
 0 & 1\\
\end{array}
\right)
\end{equation}
of $\Gamma(4)$ in \eqref{sdf2}. Hence
\begin{equation}\label{gg23}
g_1^N=\left(
\begin{array}{cc}
 1 & 4N\\
 0 & 1\\
\end{array}
\right)
\end{equation}
belongs to $\ker U_\alpha$ and is obviously primitive. Thus Wohlfahrt's level $n(\alpha)$ of $\ker U_\alpha$ is given for $\alpha$ with $N(\alpha)=N$ by 
\begin{equation}\label{fg56}
 n(\alpha)=4N.
\end{equation}

Next we use a formula due to M. Newman \cite{Newman} to determine the genus of $\ker U_\alpha$. 
Namely, let $\Gamma$ be a normal subgroup of ${\rm PSL}(2,\mathbb Z)$ with index $\mu$, genus $g$, number of its cusps $p$ and Wohlfahrt level $n$. If  $t:=\dfrac{\mu}{n}$,  
then   the following identity holds \cite{Newman}, \cite{Alexei} (page 160)
\begin{equation}\label{newmanform.}
 g=1+\dfrac{\mu}{12}-\dfrac{t}{2}.
\end{equation}
For the group $\ker U_\alpha$ we get with  \eqref{indexalpha} and \eqref{fg56}  $t=6N^2$. 
Inserting this and \eqref{indexalpha} into \eqref{newmanform.} we obtain for the genus $g(\alpha)$ of $\ker U_\alpha$
\begin{equation}
 g(\alpha)=1+2N^3-3N^2.
\end{equation}

respectively with \eqref{ggg} for the number $p(\alpha)$ of cusps of $\ker U_\alpha$ 
\begin{equation}
p(\alpha)=6N^2.
\end{equation}

Summarizing we therefore have proved for the group $\ker U_\alpha$ 
\begin{theorem}\label{theorem:inf}
If $N(\alpha)=N\in \mathbb N$ denotes the order of the generators of the abelian group $A_\alpha= \Gamma(4)/\ker U_\alpha$ with $|A_\alpha|= N^3$, let $\mu(\alpha)$ be the index of the group $\ker U_\alpha$ in ${\rm PSL}(2,\mathbb Z)$, $g(\alpha)$ its genus, $p(\alpha)$ the number of its cusps, $n(\alpha)$ its Wohlfahrt level and $\mathcal N(\alpha)$ the number of its free generators.
Then 
\begin{itemize}
 \item $\mu(\alpha)=24N^3$
 \item $g(\alpha)=1+2N^3-3N^2$
 \item $p(\alpha)=6N^2$
 \item $n(\alpha)=4N$
 \item $\mathcal N(\alpha)=4 N^3+1$
\end{itemize}    
\end{theorem}
Since  the Wohlfahrt level of $\ker U_\alpha$ is given by $n=4N(\alpha)$, one finds in case $\ker U_\alpha$  is a congruence group 
\begin{equation}\label{lll}
\Gamma(4N)\leq\ker U_\alpha. 
\end{equation}
From this one determines easily those values of $\alpha$ for which $\ker U_\alpha$ is indeed a congruence subgroup. Since the index of $\Gamma(4N)$ in ${\rm PSL}(2,\mathbb Z)$, given by
\begin{equation}
\left[{\rm PSL}(2,\mathbb Z):\Gamma(4N)\right]=\dfrac{1}{2}(4N)^3\prod_{p\mid 4N}(1-\dfrac{1}{p^2}),
\end{equation}
 must then be larger or equal to the index $\mu(\alpha)$ of $\ker U_\alpha$ in ${\rm PSL}(2,\mathbb Z)$,  one finds
\begin{equation}
\dfrac{1}{2}(4N)^3\prod_{p\mid 4N}(1-\dfrac{1}{p^2})\geq24N^3
\end{equation}
or
\begin{equation}
\dfrac{4}{3}\prod_{p\mid 4N}(1-\dfrac{1}{p^2})\geq1.
\end{equation}
Obviously this inequality holds if and only if $N=2^k, \, 0\leq k<\infty$. 
\begin{lemma}\label{congruence} If $N(\alpha)=2^k$ and $\ker U_\alpha$ is a congruence group then $\ker U_\alpha = \Gamma(2^{k+2})$ and hence $A_\alpha\cong\Gamma(4)/\Gamma(2^{k+2})$.
\end{lemma}
\begin{proof}
For $\alpha$ with $N(\alpha)=2^k$ the group $A_\alpha$ has order $2^{3k}$. If  $\ker U_\alpha$ is a congruence subgroup then $\Gamma(2^{k+2})\leq \ker U_\alpha$. On the other hand one finds for the index $\left[\Gamma(4):
\Gamma(2^{k+2})\right]$ by a simple calculation $\left[\Gamma(4):\Gamma(2^{k+2})\right]= 2^{3k}$. But $\Gamma(4)/\Gamma(2^{k+2})\geq \Gamma(4)/\ker U_\alpha \cong A_\alpha$ and hence $2^{3k}=| A_\alpha|\leq |\Gamma(4)/\Gamma(2^{k+2})|=2^{3k}$. Therefore 
$\ker U_\alpha = \Gamma(2^{k+2})$. 
\end{proof}
Next we show, that only for $k=0,1,2$ the principal congruence subgroup $\Gamma(2^{k+2})$ , that means, only $\Gamma(4), \Gamma(8)$ and $\Gamma(16)$ can coincide with the group $\ker U_\alpha$. This follows immediately from the following lemma
\begin{lemma}
The group $\Gamma(4)/\Gamma(2^{k+2})$ is abelian iff k=0,1,2.
\end{lemma}
Since we did not find this result, which is presumably well known, in the literature, we give a simple proof.
\begin{proof}
For $h_i=\begin{pmatrix} 1+4 a_i& 4 b_i\\ 4 c_i&1+4 d_i\end{pmatrix}\in \Gamma(4),\, i=1,2,$ one finds for $h_{i,j}:= h_i h_j, \, i, j =1,2$: $$ h_{1,2}= h_{2,1}= \begin{pmatrix} 1+4 (a_1+a_2)& 4 (b_1+b_2)\\ 4 (c_1+c_2)&1+4 (d_1+d_2) \end{pmatrix}\mod 16$$ respectively $$h_{1,2}^{-1} h_{2,1}=\begin{pmatrix} 1+4 (a_1+a_2+d_1+d_2)& 0\\ 0&1+4 (a_1+a_2+d_1+d_2) \end{pmatrix}\mod 16.$$ But $4| (a_i+d_i), i=1,2$, and therefore $h_{1,2}= h_{2,1} \mod \Gamma (16)$. Hence $\Gamma(4)/\Gamma(2^{k+2})$ is abelian for $k=0,1,2$. To show that this group is not abelian for $k\geq 3$ take the two elements $h_1=\begin{pmatrix} 1& 4\\ 0&1\end{pmatrix}$ respectively $h_2=\begin{pmatrix} -1& 0\\ 4 &-1 \end{pmatrix}\in \Gamma(4)$. Then one finds $h_{1,2}^{-1} h_{2,1}=\begin{pmatrix} 17& 64 \\ 64&241\end{pmatrix}$ which does not belong to $\Gamma(2^{k+2)})$ for $k\geq 3$.
\end{proof}
Next we show, that $\Gamma(16)$ cannot be a subgroup of $\ker U_\alpha$.  Assume this is the case. By lemma \ref{congruence}  $\ker U_\alpha = \Gamma(16) $.  But according to (\ref{gen}) $U_\alpha(g_5)=U_\alpha(g_3)$ for the generators $g_3$ and $g_5$ of $\Gamma(4)$ in (\ref{sdf2}) and hence $g_3^{-1}g_5 \in \ker U_\alpha$. But 
$g_3^{-1} g_5=\begin{pmatrix} 1& 8\\ 8&1\end{pmatrix} \mod 16$ which does not belong to $\Gamma(16)$. Hence $\ker U_\alpha > \Gamma(16)$, which is a contradiction.
This proves
\begin{corollary}
The group $\ker U_\alpha$ can be a congruence group only for $N(\alpha)= 1 ,2 .$
\end{corollary}
Let then  $\alpha_1$ and $\alpha_2$ denote the $\alpha$-values for which $N(\alpha_1)=1$ and $N(\alpha_2)=2$, respectively. We are going to prove $\ker U_{\alpha_1}$ and $\ker U_{\alpha_2}$ are indeed congruence groups. To this end recall that $\Gamma(4)/\ker U_\alpha\cong 
A_\alpha$. Since $A_{\alpha_1}$ is the trivial group, $\ker U_{\alpha_1}=\Gamma(4)$  and hence is a congruence group. \newline
It remains to prove the congruence property of $\ker U_{\alpha_2}$. Since $N(\alpha_2)=2$ and 
$U_{\alpha_2}(\Gamma(4))=A_{\alpha_2}$, it follows that $(U_{\alpha_2}(g))^2= id$ for all $g\in \Gamma(4)$ and hence $g^2\in \ker U_{\alpha_2}$ for all $g\in \Gamma(4)$. But $g^2\in \Gamma(8)$ for $g\in \Gamma(4)$. Therefore also the group $\langle g^2, g\in 
\Gamma(4)\rangle$ generated by $\Gamma(4)^2$ belongs to $\ker U_{\alpha_2}$ and hence $\ker U_{\alpha_2}\cap \Gamma(8)\not=\emptyset$.  Next we will show that the groups $\Gamma(8)$ and $\ker U_{\alpha_2}$  coincide. To this end we 
note that $A_{\alpha_2}\cong C_2\times C_2\times C_2$ where $C_2$ is the cyclic group of order $2$. But $A_{\alpha_2}\cong \Gamma(4)/\ker U_{\alpha_2}$ under the following well known natural group isomorphism $\imath_1:\Gamma(4)/\ker U_{\alpha_2}\rightarrow A_{\alpha_2}$: 
\begin{equation}
\imath_1( g \ker U_{\alpha_2})= U_{\alpha_2}(g).
\end{equation} 
Thereby the generators $A_i(\alpha_2), 1\leq i \leq 3$ of the group $A_{\alpha_2}$ in (\ref{genA1})-(\ref{genA3}) are mapped to the generators $g_i \ker U_{\alpha_2}, 1\leq i \leq 3$ of the group $\Gamma(4)/\ker U_{\alpha_2}$ with the $\{g_i\}$ as given in (\ref{sdf2}). Indeed, from equation (\ref{gen}) it follows that $g_3=g_5\mod \ker U_{\alpha_2}$ and $g_4=g_2^{-1} g_1^{-1} \mod \ker U_{\alpha_2}$.
On the other hand, it is known  \cite{McQuillan}, that $\Gamma(4)/\Gamma(8)$ is also isomorphic to $C_2\times C_2\times C_2$. Indeed, the elements $g_i \Gamma(8), 1\leq i\leq 3$ with $\{g_i, 1\leq i\leq 5\}$ defined in (\ref{sdf2}), are generators of the group $\Gamma(4)/\Gamma(8)$: we know that the five elements $g_i,\, 1\leq i\leq 5,$ generate the group $\Gamma(4)$ and fulfill $g_i^2= id \mod \Gamma(8)$. Furthermore one checks easily that $g_3=g_5\mod \Gamma(8)$ and $g_4=g_2^{-1} g_1^{-1}\mod \Gamma(8)$. Therefore the following map of their generators defines an  isomorphism $\imath$ of the two groups $\Gamma(4)/\ker U_{\alpha_2}$ and $\Gamma(4)/\Gamma(8)$
\begin{equation}\label{iso-factor}
\imath:\Gamma(4)/\ker U_{\alpha_2}\rightarrow\Gamma(4)/\Gamma(8)
\end{equation}
defined by 
\begin{equation}\label{isom.1}
\imath(g_i\ker U_{\alpha_2})=g_i\Gamma(8)
\end{equation}
Indeed $\imath (g_i\ker U_{\alpha_2}  g_j\ker U_{\alpha_2})=\imath (g_1 g_2\ker U_{\alpha_2})=g_i g_j \Gamma(8)=g_i \Gamma(8)g_j \Gamma(8)$. Since any $g\in \Gamma(4)$ can be expressed both $\mod \ker U_{\alpha_2}$ and $\mod \Gamma(8)$ in terms of the generators $g_i, 1\leq i\leq 3$, this implies for all $g\in\Gamma(4)$ $\imath(g\ker U_{\alpha_2})=g\Gamma(8)$. For $g\in \ker U_{\alpha_2}$ this implies necessarily $g\in \Gamma(8)$, that means $\Gamma(8)\geq \ker U_{\alpha_2}$. Similar arguments as in lemma \ref{congruence} then imply, that the two groups must coincide. 
This shows
\begin{corollary}
The kernel $\ker U_{\alpha_2}$ is given by $\Gamma(8)$ and hence $U_{\alpha_2}$ is a congruence representation.
\end{corollary}
From the definition of the generators of $A_\alpha$ in \eqref{genA1}, \eqref{genA2}, and \eqref{genA3} it is clear that $N(\alpha_1)=1$ iff
$8\pi i\alpha_1=2\pi i k$ iff $\alpha_1=(1/4) k$ with $k\in \mathbb Z$. Moreover, $N(\alpha_2)=2$ iff $8\pi i\alpha_2=\pi i k$ iff $\alpha_2=(1/8) k$ with $k\in \mathbb Z$ and $(k,2)=1$.

Summarizing our discussion of the congruence properties of the kernels $\ker U_\alpha$ we have 
\begin{theorem}\label{jkjkp}
The representation $U_\alpha, 0\leq\alpha\leq1/2$, defined in \eqref{rep.-U}  is congruence  only for the $\alpha$-values $0,\frac{1}{8},\frac{2}{8},\frac{3}{8},\frac{4}{8}$. Moreover we have
\begin{equation}
\ker U_0=\ker U_{\frac{2}{8}}=\ker U_{\frac{4}{8}}=\Gamma(4),
\end{equation}
respectively
\begin{equation}
\ker U_{\frac{1}{8}}=\ker U_{\frac{3}{8}}=\Gamma(8).
\end{equation}
\end{theorem}
This obviously implies the well known result of Newman et al. on the congruence properties of the character $\chi_\alpha$. Contrary to the latter case, where the principal congruence groups $\Gamma(2d),\,d=1,2,4,8$ appear as subgroups for the congruence character $\chi_
\alpha$, for the induced representation $U_\alpha$ only the two groups $\Gamma(4)$ and $\Gamma(8)$ are related to the congruence properties of its representations. The resulting noncongruence groups on the other hand  are of completely different nature in the two cases. It would be interesting to see if something similar happens also for the induced representations of other characters, like for instance the $\chi_n$  studied in \cite{LLY}, \cite{ALL} and \cite{Long} for the congruence group $\Gamma^1(5)$. 
\bibliography{congruence_2}{}
\bibliographystyle{plain}
\end{document}